\documentclass[reqno,12pt]{amsart}
\makeatletter

\usepackage[T1]{fontenc}
\usepackage{lmodern}
\usepackage{appendix}
\usepackage{mathtools}
\usepackage{fullpage}
\usepackage{textcmds}  
\usepackage{amsmath, amssymb, amsfonts, amstext, verbatim, amsthm, mathrsfs, stmaryrd}
\usepackage{microtype}
\usepackage[all,cmtip,2cell]{xy}
\usepackage{pgf,tikz,pgfplots}
\pgfplotsset{compat=1.15}
\usetikzlibrary{arrows}
\usetikzlibrary{positioning}
\usetikzlibrary{quotes}
\usepackage{tikz-cd}
\usepackage[scr=boondox,  
            cal=esstix]   
           {mathalpha}
\tikzset{
  symbol/.style={
    draw=none,
    every to/.append style={
      edge node={node [sloped, allow upside down, auto=false]{$#1$}}}
  }
}

\tikzcdset{scale cd/.style={every label/.append style={scale=#1},
    cells={nodes={scale=#1}}}}

\usepackage{enumerate}
\usepackage{enumitem}
\usepackage[colorlinks=true,linkcolor=blue,citecolor=blue,urlcolor=blue,citebordercolor={0 0 1},urlbordercolor={0 0 1},linkbordercolor={0 0 1},pagebackref]{hyperref} 
\usepackage[alphabetic]{amsrefs} 
\usepackage[nameinlink]{cleveref}
\crefname{defn}{definition}{definitions}
\Crefname{defn}{Definition}{Definitions}

\usepackage{enotez}
\setenotez{backref=true}
\usepackage{float}
\usetikzlibrary{calc}

\def\makeCal#1{%
\expandafter\newcommand\csname c#1\endcsname{\mathcal{#1}}}
\def\makeBB#1{%
\expandafter\newcommand\csname b#1\endcsname{\mathbb{#1}}}
\def\makeFrak#1{%
\expandafter\newcommand\csname f#1\endcsname{\mathfrak{#1}}}

\count@=0
\loop
\advance\count@ 1
\edef\y{\@Alph\count@}%
\expandafter\makeCal\y
\expandafter\makeBB\y
\expandafter\makeFrak\y
\ifnum\count@<26
\repeat

\theoremstyle{plain}
\newtheorem{thm}{Theorem}[section]
\newtheorem{cor}[thm]{Corollary}
\newtheorem{lem}[thm]{Lemma}

\newtheorem{prop}[thm]{Proposition}

\theoremstyle{definition}
\newtheorem{rem}[thm]{Remark}

\newtheorem{ex}[thm]{Example}

\def\rm{\mathrm}

\DeclareMathOperator{\Bl}{Bl}
\newcommand{\dual}{\vee}

\DeclareMathOperator{\Chow}{Chow}

\DeclareMathOperator{\Spec}{Spec}

\newcommand{\sh}{\mathcal}
\newcommand{\spec}{\mathrm{Spec}}

\newcommand{\arr}{\mathcal}

\def\bb{\mathbb}

\usepackage{pbox}
\usepackage[normalem]{ulem}

\newcommand\blfootnote[1]{%
  \begingroup
  \renewcommand\thefootnote{}\footnote{#1}%
  \addtocounter{footnote}{-1}%
  \endgroup
}

\newcommand{\bcg}{{\mathcal B}}
\newcommand{\poly}{{\mathcal{Poly}}}
\newcommand{\polyh}{\mathcal{Poly}_\infty}
\newcommand{\diagh}{\mathcal{Diag}_\infty}

\makeatletter

\usepackage{babel}

\begin{document}

\address{Prabhat Devkota \newline
\indent Stony Brook University, Department of Mathematics\newline
\indent 100 Nicolls Road, Stony Brook, NY, 11794}
\email{prabhat.devkota@stonybrook.edu}

\address{Antonios-Alexandros Robotis \newline
\indent Cornell University, Department of Mathematics\newline
\indent 212 Garden Avenue, Ithaca, NY, 14853}
\email{ar2377@cornell.edu}

\address{Adrian Zahariuc \newline
\indent University of Windsor, Department of Mathematics and Statistics\newline
\indent 401 Sunset Avenue, Windsor, ON, N9B 3P4, Canada}
\email{adrian.zahariuc@uwindsor.ca}

\title{Multiscale differentials and wonderful models}

\author[P. Devkota]{Prabhat Devkota}
\author[A. Robotis]{Antonios-Alexandros Robotis}
\author[A. Zahariuc]{Adrian Zahariuc}

\begin{abstract}
    We study the relationships between several varieties parametrizing marked curves with differentials in the literature. More precisely, we prove that the space $\cB_n$ of multiscale differentials of genus $0$ with $n+1$ marked points of orders $(0,\ldots, 0, -2)$, as introduced in \cite{bcggm}, is a wonderful variety in the sense of \cite{bhmpw2022semismalldecompositionofchowringofmatroid}. This shows that the Chow ring of $\cB_n$ is generated by the classes of a collection of smooth boundary divisors with normal crossings subject to simple and explicit linear and quadratic relations. Furthermore, we realize $\cB_n$ as a subvariety of the space $\cA_n$ of multiscale lines defined in \cite{augmentedstability} and prove that $\cB_n$ can be realized as the normalized Chow quotient of $\cA_n$ by a natural $\bC^*$-action.
\end{abstract}

\maketitle

\tableofcontents

\section{Introduction}
\label{S:introduction}

Configuration spaces of points and their compactifications are classical objects of study, with a myriad of applications in both mathematics and physics. A classical example of this type is the space of configurations of $n$ distinct points in $\bC$ considered up to affine linear automorphisms of $\bC$. This admits a rather remarkable compactification known as the Grothendieck-Knudsen space of stable $n$-pointed genus $0$ curves $\overline{M}_{0,n+1}$, which is a smooth $(n-2)$-dimensional complex projective manifold. These spaces and their higher genus versions, the Deligne-Mumford stacks $\overline{\cM}_{g,n}$\footnote{We employ the convention of denoting moduli stack of curves by calligraphic $\overline{\cM}_{g,n}$; in genus 0, we stick to roman $\overline{M}_{0,n}$ since it is a smooth scheme.} parametrizing stable genus $g$ curves with $n$ marked points, have been the subject of intense study -- cf. \cites{DM69,Knudsen1983}.

To better understand a geometric space one often studies the various bundles defined over it. The moduli stack $\cM_{g,n}$ of nonsingular genus $g$ curves with $n$ marked points is no different. Of particular interest are the twisted Hodge bundles. Given an $n$-tuple of integers $\mu=(m_1,\ldots,m_n)$ with $m_1\geq\cdots\geq m_k\geq 0 >m_{k+1}\geq\cdots \geq m_n$ satisfying $\sum_im_i=2g-2$, denote by $\tilde{\mu}=(m_{k+1},\ldots,m_n)$ the negative part of $\mu$. Then the associated twisted Hodge bundle $\bE_{g,n}(\tilde{\mu})$ is a vector bundle over $\cM_{g,n}$ whose fiber over a closed point $(C;x_1,\ldots,x_n)$ is the vector space $H^0(C,\omega_C(-\sum_{i\geq k+1}m_ix_i))$ of meromorphic 1-forms with at worst prescribed poles at $x_i$.

Inside the twisted Hodge bundle are certain special subspaces, called the moduli spaces of \emph{meromorphic differentials} $\Omega\mathcal{M}_{g,n}(\mu)$. These are subspaces whose fiber over $(C;x_1,\ldots,x_n)$ is the collection of meromorphic differentials $\omega$ on $C$ having zero (or pole) of order $m_i$ at each $x_i$. Because of their interplay with $\cM_{g,n}$, these spaces have garnered considerable interest in algebraic geometry and other areas. For instance, because a meromorphic differential endows the underlying Riemann surface with the structure of a (singular) flat metric, these moduli spaces of abelian differentials have been objects of particular interest in Teichm\"uller dynamics -- see \cite{chen2016teichmullerdynamicseyesalgebraic} for the discussion of Teichm\"uller dynamics from an algebro-geometric perspective.

The moduli space of abelian differentials $\Omega\mathcal{M}_{g,n}(\mu)$ admits a $\bC^*$-action by rescaling the differentials; the resulting quotient, denoted $\bP\Omega\mathcal{M}_{g,n}(\mu)$ and also usually called the moduli space of abelian differentials, admits a modular compactification $\bP\Xi\overline{\mathcal{M}}_{g,n}(\mu)$ parametrizing \emph{multiscale differentials} on a stable $n$-pointed curve of genus $g$ \cite{bcggm}. A multiscale differential $(C;x_1,\ldots,x_n;\omega;\sigma;\Gamma)$ consists of the following data:
\begin{enumerate}
    \item a stable $n$-pointed genus $g$ curve $(C;x_1,\ldots,x_n)\in\overline{\mathcal{M}}_{g,n}$;\vspace{2mm}
    \item a level graph $\Gamma$ with $L$ levels compatible with the dual graph of $(C;x_1,\ldots,x_n)$;\vspace{2mm}
    \item an assignment of a non-negative integer on each edge of $\Gamma$, called enhancement, that specifies the order of zeros/poles on the nodes;\vspace{2mm}
    \item a collection $\omega=\{\omega_v\}_{v\in V(\Gamma)}$ of meromorphic differentials, one for each vertex on $\Gamma$, with prescribed zeros and poles at the marked points and the nodes, satisfying a \emph{Global Residue Condition}; and 
    \vspace{2mm}
    \item a collection $\sigma$ of \emph{prong-matchings} on the edges of $\Gamma$.
\end{enumerate}

Neither the global residue condition nor prong matching will play any role in our discussion, so we refer to \cite{bcggm} for more details. Similarly, the order of zero/pole on every node is determined by the underlying dual graph if the dual graph is a tree \cite{Devkotacohomology}*{Rem. 1}, so we will also refer to \cite{bcggm} for the discussion on enhancements. The moduli space $\bP\Xi\overline{\mathcal{M}}_{g,n}(\mu)$ parametrizes equivalence classes of multiscale differentials, where equivalent multiscale differentials differ by the action of a torus $(\bC^*)^L$ that acts by rescaling the differentials level-by-level. This moduli space of multiscale differentials is a smooth and proper Deligne--Mumford stack with a normal crossing boundary divisor \cite{bcggm}, whose coarse moduli space is a projective variety \cite{chen2022kodairadimensionmodulispaces}*{\S3}. Forgetting the differentials, we obtain a morphism 
    \[ \tau:\bP\Xi\overline{\mathcal{M}}_{g,n}(\mu)\rightarrow \overline{\mathcal{M}}_{g,n}, \] which is generically one-to-one, but not surjective. However, when $g=0$, $\tau$ is surjective and thus birational at the level of coarse moduli spaces. Throughout this paper, unless otherwise stated, we consider genus $0$ curves with $n+1$ marked points of orders $\mu=(0,\ldots,0,-2) $, so let us fix the notation 
    \[ \bcg_n\coloneqq\bP\Xi\overline{M}_{0,n+1}(0,\ldots,0,-2). \] 
    As noted in \cite{Devkotacohomology}*{\S4}, $\bcg_n$ is actually a smooth {\em projective variety}.

    \subsection{The space \texorpdfstring{$\bcg_n$}{Bn} as a wonderful variety} Our first result is that $\bcg_n$ coincides with a long-known space which is well-studied both combinatorially and geometrically. We will denote this space by $Y_n$. It can be described equivalently as:
	\begin{enumerate}
		\item the maximal wonderful model \cites{deconciniprocesi1995wonderfulmodels, li2009wonderfulcompactificationarrangementsubvarieties} of the projectivization of the $A_{n-1}$ Coxeter arrangement; \vspace{2mm}
		\item the wonderful variety \cite{bhmpw2022semismalldecompositionofchowringofmatroid}*{Rem. 2.13} associated to the matroid $M(K_n)$ of the complete graph $K_n$ and its representation by a positive root system of $\mathfrak{sl}_n$;  \vspace{2mm}
		\item the fiber of Ulyanov's polydiagonal compactification $C\langle n \rangle$ \cite{ulyanov2002polydiagonalcompactification} over a point on the small diagonal of $C^n$, where $C$ is a smooth curve.
	\end{enumerate}
    
	More explicitly, consider $n$ points $p_1,\ldots,p_n \in \bP^{n-2}$ in general linear position, take all proper linear subspaces of $\bP^{n-2}$ spanned by a subset of $\{ p_1,\ldots,p_n \}$ and all possible intersections of all these spans, and blow up all these linear loci in increasing order of dimension. The equivalence of the first two definitions is straightforward after unwinding them -- see \Cref{lem: Y is wonderful variety}. We will not use the equivalence with the third definition, though it can be proved with some work using Ulyanov's construction -- see \cite{li2009wonderfulcompactificationarrangementsubvarieties}*{Prop. 2.8} and \cite{zahariuc2024smallresolutionsmodulispaces}*{Prop. B.1}.

    \begin{thm}[\Cref{multiscale_wonderful}]\label{intro_multiscale_wonderful}
        There is an isomorphism $\bcg_n \cong Y_n$.
    \end{thm}

    It follows that $\mathrm{CH}^*(\bcg_n)$ is isomorphic to the Chow ring of the matroid $M(K_n)$ -- see \cite{feichtneryuzvinskychowring}*{Cor. 2}. In particular, it is generated by the classes of combinatorially described smooth boundary divisors subject to linear and quadratic relations -- see \Cref{C:chowring2}. We were informed by Johannes Schmitt that this description of the Chow ring also follows from \cite{pandharipande2025logarithmictautologicalringsmoduli}*{Thm. 11}.

    \subsection{Relation to the space of multiscale lines with collisions} A feature of the spaces introduced in \cite{bcggm}, including $\bcg_n$, is that bubbling occurs when one takes the limit of families of curves where marked points become arbitrarily close. This resembles the behavior of $\overline{M}_{0,n+1}$ where colliding marked points produce new irreducible components. On the other hand, in certain contexts it is necessary to find a modular compactification of the space of $n$ not necessarily distinct points in $\bb{C}$ up to translation; i.e. $\bC^n/\bC$, where $\bC$ is embedded in $\bC^n$ as the small diagonal. This is one of the objectives of \cite{augmentedstability}, where $\bC^n/\bC$ appears as the local model of the quotient of the manifold of Bridgeland stability conditions \cite{Br07} by $\bC$. The problem of compactifying the stability manifold is thereby related to the problem of producing a suitable compactification of $\bC^n/\bC$.

    An element of $\bC^n/\bC$ can be equivalently regarded as an isomorphism class of data $(\bP^1,\infty,\omega,p_1,\ldots, p_n)$, where $\omega$ is a meromorphic differential on $\bP^1$ with a unique order $2$ pole at $\infty$, and $p_1,\ldots, p_n$ are points in $\bC=\bP^1\setminus \{\infty\}$. That is, the marked points $p_1,\ldots, p_n$ are allowed to collide with each other but not with $\infty$.\footnote{This is reminiscent of the situation of \cite{Hassett} but not a special case. Indeed, the numerical constraints in the genus $0$ case in \emph{loc. cit.} do not allow for ($n+1$)-marked genus $0$ curves where $n$ points can collide with each other but not with the remaining marked point.} This perspective leads to the construction of the \emph{moduli space of multiscale lines with collision} $\cA_n$, as in \cite{augmentedstability}. An $n$\emph{-marked multiscale line} consists of data $(\Sigma,\preceq,\omega_\bullet,p_\infty,p_1,\ldots, p_n)$, where 
    \begin{enumerate}
        \item $\Sigma$ is a nodal genus $0$ curve equipped with marked points $p_\infty,p_1,\ldots, p_n$ such that $p_i\ne p_\infty$ for all $i=1,\ldots,n$; \vspace{2mm}
        \item $\preceq$ is a level structure on the set of irreducible components of $\Sigma$; i.e. a relation which determines when two vertices of the dual tree are at the same ``height'';\footnote{In particular, the dual tree $\Gamma(\Sigma)$ of $\Sigma$ is a level tree as in \cite{bcggm}.} and \vspace{2mm}
        \item $\omega_\bullet$ is a collection of meromorphic differentials, one per irreducible component of $\Sigma$, with a unique pole of order $2$ on their respective components at the node closest to $p_\infty$ (resp. $p_\infty$ in the root case).
    \end{enumerate}
    The level tree $(\Gamma(\Sigma),\preceq)$ has a unique top component, which is the one containing $p_\infty$. Furthermore, the marked points $p_1,\ldots, p_n$ are required to lie on the bottom components of $\Sigma$, all of which are at the same level. The level structure endows the dual tree $\Gamma(\Sigma)$ of $\Sigma$ with the structure of a level graph as in the case of \cite{bcggm} mentioned above. The space $\cA_n$ constructed in \cite{augmentedstability} parametrizes $n$-marked multiscale lines up to \emph{complex projective equivalence}, a variant of the equivalence considered in \cite{bcggm} except that bottom level differentials are not rescaled. The space $\cA_n$ is a proper algebraic variety by \cite{augmentedstability}*{Thm. 4.12} which is projective by \cite{robotis2024spaces}*{Thm. 4.7}, where it is furthermore shown that it coincides with an augmented wonderful variety as studied in \cite{bhmpw2022semismalldecompositionofchowringofmatroid}. 
    
    The locus of irreducible curves $\cA_n^\circ\subset \cA_n$ consists of points $(\bP^1,\omega,p_\infty,p_1,\ldots, p_n)$ as above and is thus identified with $\bC^n/\bC$. Since $\omega \in \Omega_{\bP^1}(2p_\infty) \cong \cO_{\bP^1}$, it is determined uniquely up to a scalar multiple. By contrast, the locus of irreducible curves in $\bcg_n$ consists of the same data, except that $\omega$ is only defined up to $\bC^*$-scaling. Consequently, there is a map from the open subset $\cA_n^{\circ\circ} \subset \cA_n$ consisting of irreducible curves with non-colliding marked points to $\cB_n$, with $\bC^*$-fiber corresponding to the forgotten differential $\omega$, and we obtain a rational map
    \[ 
        \phi_n:\cA_n \dashrightarrow \bcg_n. 
    \]
    
    These observations, as well as many geometric similarities between the spaces $\cA_n$ and $\cB_n$, motivate the present investigation of their precise relationship. We show that there is a closed immersion $\bcg_n \hookrightarrow \cA_n$ realizing $\bcg_n$ as component of the boundary divisor of $\cA_n$ -- see \Cref{C: special divisor2}. This can be regarded as a special case of the fact that a wonderful variety is a divisor on the respective augmented wonderful variety \cite{bhmpw2022semismalldecompositionofchowringofmatroid}*{Rem. 2.13}, in light of \Cref{intro_multiscale_wonderful} and \cite{robotis2024spaces}*{Thm. 4.7}. First, we construct a natural resolution of $\phi_n$.

    \begin{prop}\label{prop: resolution of A to B map}
        The rational map $\phi_n$ is regular in a Zariski neighborhood of $\bcg_n$, regarded as a divisor on $\cA_n$. Moreover, $\phi_n$ admits a resolution $\cR \to \bcg_n$ which factors through the $\bP^1$-bundle $P\to \bcg_n$, whose fiber over any point in $\bcg_n$ is the compactified (co)tangent line $\bP(\bC \oplus T^\vee_{x_{n+1}}C)$ at the point of order $-2$. More precisely, there is a commutative diagram
        \begin{center}
            \begin{tikzpicture}[scale = 1.8]
            \node (sw) at (0,0) {$\cA_n$};
            \node (nw) at (0,1) {$\cR$};
            \node (ne) at (1,1) {$P$};
            \node (se) at (1,0) {$\bcg_n$};
            \draw [->] (nw) to (sw); \draw [->] (nw) to (ne); \draw [->] (ne) to (se); \draw [->] (nw) to (se);
            \draw [->, dashed] (sw) to node [anchor=south] {$\phi_n$} (se);
            \end{tikzpicture}
        \end{center}
        such that $\cR \to \cA_n$ is a composition of blowups with smooth centers disjoint from $\bcg_n \hookrightarrow \cA_n$.
    \end{prop}
    \noindent
    The variety $\cR$ is constructed explicitly as a wonderful model in \Cref{Section 2}.
    
    Regarded as a subvariety of $\cA_n$, $\cB_n$ is contained in the fixed locus of the $\bC^*$-action which scales the differentials on the bottom level of a multiscale line. Consequently, the map $\cA_n^{\circ\circ} \to \cB_n$ above can be realized explicitly as a $\bC^*$-retraction of $\cA_n^{\circ\circ}$ onto an open subset of $\cB_n$. In \Cref{Section 4}, we strengthen these observations by realizing $\cB_n$ as both the normalized Chow quotient of $\cA_n$ under the $\bC^*$-action and the geometric invariant theory quotient for a suitable choice of linearization.
     
    \begin{thm}[\Cref{thm: normalized Chow quotient}]
        There is an isomorphism $\bcg_n \cong \cC(\cA_n)$, where $\cC(\cA_n)$ is the normalized Chow quotient of $\cA_n$ by $\bC^*$.
    \end{thm}

    See \cite{ChowquotientCstar} for a recollection of normalized Chow quotients by $\mathbb{C}^*$-actions.

    \begin{thm}[\Cref{thm: GIT quotient}]
        There is an ample line bundle $\mathcal{L}$ on $\cA_n$ that admits a $\bC^*$-linearization such that the resulting GIT quotient $\cA_n \sslash_{\mathcal{L}}\bC^*$ is isomorphic to $\bcg_n$.
    \end{thm}

Even when $\mu \ne (0^{n},-2)$, the moduli spaces of genus $0$ multiscale differentials $\bP\Xi\overline{M}_{0,n+1}(\mu)$ have many combinatorial similarities to $\bcg_n$, especially in regards to the stratification of the boundary. In particular, when $\mu=(-a,b_1,\ldots,b_n)$ or $(a,-b_1,\ldots,-b_n)$ for $a,b_i>0$, the boundary strata of these spaces are also in one-to-one correspondence with stable rooted level trees where the unique zero or the unique pole is contained in the root vertex. The main difference is that unlike $\bcg_n$, these spaces are smooth Deligne-Mumford stacks; consequently, their coarse moduli spaces can have orbifold singularities. 

In spite of these differences, we anticipate analogous results connecting the intersection theory on the moduli spaces $\bP\Xi\overline{\cM}_{0,n+1}(\mu)$ to the combinatorics of matroids or a suitable stacky refinement. Conversely, it is interesting to ask which wonderful models in the sense of \cite{deconciniprocesi1995wonderfulmodels} can be realized as moduli spaces of marked genus $0$ curves with differentials, with examples being provided thus far by $\overline{\cM}_{0,n+1}$ in \cite{KapranovChow}, $\cA_n$ in \cites{robotis2024spaces,zahariuc2024smallresolutionsmodulispaces}, and $\cB_n$ in the present work. In future work, we hope to strengthen this bridge between combinatorics and moduli spaces of curves by further studying the case of $\bP\Xi\overline{M}_{0,n+1}(\mu)$.

\subsection*{Acknowledgements} P.D. thanks his advisor Samuel Grushevsky for helpful discussions and support. The research of P.D was supported in part by NSF grant DMS-21-01631. A.R. thanks Daniel Halpern-Leistner for his support and for helpful conversations concern\-ing aspects of this paper and Maria Teresa for her love and encouragement. A.Z. was partially supported by an NSERC Discovery Grant.\blfootnote{\includegraphics[scale=0.3]{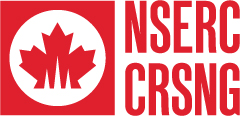} We acknowledge the support of the Natural Sciences and Engineering Research Council of Canada (NSERC), RGPIN-2020-05497. Cette recherche a \'{e}t\'{e} financ\'{e}e par le Conseil de recherches en sciences naturelles et en g\'{e}nie du Canada (CRSNG), RGPIN-2020-05497.}

\section{Wonderful models}
\label{Section 2}

We first introduce several subspace arrangements in projective space that play a key role in the sequel. Let $L_n$ be the set of partitions of $[n] = \{1,2,\ldots,n\}$. The set of blocks of $\sigma \in L_n$ is denoted by $B(\sigma)$. Let $\bot$ be the partition with $1$ block, and $\top$ the partition with $n$ blocks. Given a partition $\sigma$, we denote by $\sim_\sigma$ the induced equivalence relation on $[n]$. That is, $i\sim_\sigma j$ if $i$ and $j$ are in the same block of $\sigma$. Next, let $H$ be the hyperplane in $\bP^{n-1}$ defined by $X_1+\cdots + X_n = 0$. For $\sigma \in L_n$, let 
\[ 
    \Delta_\sigma = \{ [X_1:\cdots:X_n] \in \bP^{n-1}: X_i = X_j \text{ if $i \sim_\sigma j$} \} 
\]
be the corresponding \emph{polydiagonal} of $\bP^{n-1}$, and $\Delta^h_\sigma = \Delta_\sigma \cap H$. We define a pair of subspace arrangements of $\bP^{n-1}$ by 
\[ 
    \poly = \{ \Delta_\sigma: \sigma \in L_n \setminus \{\top \} \} \quad \text{and}  \quad \polyh = \{ \Delta^h_\sigma: \sigma \in L_n \setminus \{\bot \} \}.  
\]
Similarly, for each subset $S \subseteq [n]$ with at least two elements, we have a \emph{diagonal} $\Delta_S$ defined by $X_i=X_j$ for all $i,j \in S$, and $\Delta^h_S = \Delta_S \cap H$ if $S \neq [n]$. We define another subspace arrangement by 
\[ 
    \diagh = \{ \Delta_S^h: S \subset [n], 2 \leq |S| \leq n-1 \}. 
\]

The problem of compactifying the complement $\mathfrak{A}$ of an arrangement of subspaces $\cP$ in a vector space $V$ or its projectivization $\bP(V)$ has been long studied. In \cite{deconciniprocesi1995wonderfulmodels}, the notion of a \emph{wonderful model} $Y_{\mathfrak{A}}$ of a subspace arrangement in $\bP(V)$ is introduced. The authors define $Y_{\mathfrak{A}}$ to be the closure of the image of the map 
\[
    \mathfrak{A}\to \bP(V) \times \prod_{D\in \cP}\bP(V/D)
\]
where $\mathfrak{A} \to \bP(V)$ is the inclusion and $\mathfrak{A} \to \bP(V/D)$ is the composite $\mathfrak{A}\hookrightarrow \bP(V\setminus D) \to \bP(V/D)$ for each $D\in \cP$. When $\cP$ is a \emph{building set} \cite{li2009wonderfulcompactificationarrangementsubvarieties}*{Def. 2.2}, $Y_\mathfrak{A}$ admits a description as an iterated blowup: choose a total order $\le$ refining the containment partial order on $\cP$ and enumerate $\cP$ as $\{L_1,\ldots, L_N\}$ with respect to $\le$. Next, let $Y_0 := \bP(V)$, and if $Y_i$ has been constructed, define $Y_{i+1}$ to be the blowup of $Y_i$ along the strict transform of $L_{i+1}$. 

The output of this process is a smooth variety $Y_N =  Y_{\mathfrak{A}}$ such that the elements of $\cP$ correspond to divisors in $Y_{\mathfrak{A}}$ which give a simple normal crossings boundary divisor compactifying $\mathfrak{A}$. When $\cP$ arises as a collection of subspaces coming from a linear represen\-tation of a matroid, the corresponding wonderful model is called a \emph{wonderful variety} -- see \cite{bhmpw2022semismalldecompositionofchowringofmatroid}*{Rem. 2.13}. In \cite{li2009wonderfulcompactificationarrangementsubvarieties}, wonderful models are generalized to the non-linear setting of arrangements of subvarieties of a variety. We will not need the full generality of the setup of \cite{li2009wonderfulcompactificationarrangementsubvarieties} here, but we will frequently reference the results therein. In what follows, we will consider the following wonderful models.
\renewcommand{\arraystretch}{1.4}
\begin{center}
	\begin{tabular}{c|c|c}
		wonderful model & building set & ambient space \\ \hline
		$\overline{M}_{0,n+1}$ & $\diagh$ & $H \cong \bP^{n-2}$ \\
		$Y_n$ & $\polyh \setminus \{ \Delta^h_\top \}$ & $H \cong \bP^{n-2}$ \\
		$\mathrm{Bl}_{\Delta_\bot}\bP^{n-1}$ & $\{\Delta_\bot\}$ & $\bP^{n-1}$ \\
		$\cA_n$ & $\polyh$ & $\bP^{n-1}$ \\
		$P$ & $\poly$ & $\bP^{n-1}$ \\
		$\cR$ & $\polyh \cup \poly$ & $\bP^{n-1}$ \\
	\end{tabular}
\end{center}
The fact that $\overline{M}_{0,n+1}$ is the wonderful model of $\diagh$ follows from \cite{KapranovChow}*{Thm. 4.3.3}, while the fact that $\cA_n$ is the wonderful model of $\polyh$ in $\bP^{n-1}$ follows from \cite{robotis2024spaces}*{Thm. 4.7}. The other cases are definitions. Note that, with the exception of $\diagh$, all building sets coincide with their induced arrangements; that is, the wonderful models are maximal.

\begin{lem}\label{lem: Y is wonderful variety}
	   The wonderful model $Y_n$ coincides with the \emph{wonderful variety} constructed in \cite{bhmpw2022semismalldecompositionofchowringofmatroid}*{Rem. 2.13}, when the matroid $M$ in \emph{loc. cit.} is the graphic matroid $M(K_n)$ of the complete graph, and $M(K_n)$ is represented by a positive root system of $\mathfrak{sl}_n$. 
\end{lem}

    To clarify the statement, recall that the setup from \emph{loc. cit.} involves a matroid $M$ on a set $E$, and a vector subspace $V \subseteq \bC^E$. The dual $\bC^E \twoheadrightarrow V^\dual$ maps the standard basis of $\bC^E$ to $|E|$ vectors in $V^\dual$, so the choice of the subspace $V$ corresponds to a representation of $M$.  
	
	\begin{proof} 
	We embed $\bC \hookrightarrow \bC^n$ diagonally. The section $\bC^n/\bC \to \{ x_1+\cdots+x_n = 0 \} \subset  \bC^n$ of $\bC^n \to \bC^n/\bC$ defined by sending a class to its unique representative with centroid zero gives an identification $\bP(\bC^n/\bC) \cong H$. The setup in the statement of the lemma amounts to taking $V \subset \bC^{\binom{n}{2}}$ given by $x_{\{i,j\}} - x_{\{i,k\}} + x_{\{j,k\}} = 0$ for all $1 \leq i<j<k \leq n$
    in the setup of \cite{bhmpw2022semismalldecompositionofchowringofmatroid}*{Rem. 2.13}. Then $V \cong \bC^n/\bC$, so $\bP V \cong \bP(\bC^n/\bC) \cong H$. With the notation in \emph{loc. cit.}, $M = M(K_n)$, the linear subspaces $H_F$ correspond to elements of $\polyh$, and hence $\underline{X}_V = Y_n$. This proof is analogous to the augmented case in \cite{zahariuc2024smallresolutionsmodulispaces}*{\S2.2}.
	\end{proof}

    \begin{rem}
        More intrinsically, one can take the $(n-1)$-dimensional projective space to be $\bP(V\oplus \bC)$ for $V = \bC^n/\bC$ and the hyperplane $H$ to be $\bP(V) \hookrightarrow\bP(V\oplus \bC)$ defined by $t = 0$ for $t\in (V\oplus \bC)^\vee$ dual to $0\oplus \bC$. Subject to these identifications, the elements of $\poly$ are defined by the vanishing of various linear forms $\Pi_{ij} = X_j - X_i \in (\bC^n/\bC)^\vee$ and the elements of $\polyh$ are defined similarly with the added condition of lying in $\bP(V)$, i.e. $t=0$. This more intrinsic formulation does not play a major role in this paper, however it explains the sense in which the space $\cA_n$ is a compactification of $\bC^n/\bC$ -- see \cites{robotis2024spaces,zahariuc2024smallresolutionsmodulispaces}.

        This perspective also clarifies why $\overline{M}_{0,n+1}$ is obtained by blowing up $\bP^{n-2}\cong \bP(V)$. Indeed, modulo automorphisms, a generic $(n+1)$-marked $\bP^1$ is equivalent to a configuration $(p_1,\ldots, p_n,\infty)$ with $p_i\in \bC$ for $i=1,\ldots, n$ considered up to translation and scaling -- i.e. as an element of $\bP(V)$. Upon blowing up the subspaces of $\bP(V)$ given by $\diagh$, which correspond to collision of marked points, one obtains $\overline{M}_{0,n+1}$ as in \cite{KapranovChow}.
    \end{rem}

    We will next construct a diagram as follows:
    \begin{equation}
    \label{diagram: big diagram}
        \begin{tikzcd}
            &&\bP(\cO_H\oplus \cO_H(1))\arrow[d,equal]\\
            \cR\arrow[r,"\nu"]\arrow[dr,"\lambda"]\arrow[d,"\alpha",swap]&P\arrow[r,"\gamma"]\arrow[d,"\rho"]& \Bl_{\Delta_\bot}\bP^{n-1}\arrow[d,"\varpi"]\\
            \cA_n\arrow[r,dashed,"\phi_n"]&Y_n\arrow[d,"\mu_n"]\arrow[r,"\beta"]\arrow[l,bend left,"\delta"]&H\\
            &\overline{M}_{0,n+1}\arrow[ur,"\kappa",swap]&
        \end{tikzcd}
    \end{equation}
The morphisms $\alpha,\nu,\gamma,\beta,\kappa,\mu_n$ in \eqref{diagram: big diagram}, all of which are birational, are defined using the same idea relying on \cite{li2009wonderfulcompactificationarrangementsubvarieties}*{Thm. 1.3}, as we explain now. Let us focus on $\nu$ and $\gamma$ first, for instance. If we order the building set of $\cR$ as $\Delta_\bot$, then $\poly \setminus \{ \Delta_\bot \}$ in order of increasing dimension (though there is some flexibility in these choices), followed by $\polyh$, then condition $(*)$ in \cite{li2009wonderfulcompactificationarrangementsubvarieties}*{Thm. 1.3} is satisfied. Thus, if we start with $\bP^{n-1}$ and blow up the dominant transforms in this order, we will obtain $\cR$ by \cite{li2009wonderfulcompactificationarrangementsubvarieties}*{Thm. 1.3}. However, by definition, during this process we will see $\mathrm{Bl}_{\Delta_\bot}\bP^{n-1}$ and $P$, thus giving the birational morphisms $\nu$ and $\gamma$. The other birational morphisms are obtained similarly. To construct $\alpha:\cR \to \cA_n$, we enumerate $\polyh \cup \poly$ as $\polyh$ first and $\poly$ second both in order of increasing dimension, and recall that blowing up $\polyh$ yields $\cA_n$. For $\mu_n$, $\kappa$, and $\beta$, we enumerate $\polyh \setminus \{ \Delta^h_\top \}$ as $\diagh$ first, then everything left over, and note that condition $(*)$ in \cite{li2009wonderfulcompactificationarrangementsubvarieties}*{Thm. 1.3} is still satisfied.

The map $\varpi$ is simply projection from the point $[1:\cdots:1] \in \bP^{n-1}$, which exhibits the blowup at this point as a $\bP^1$-bundle over $H$. The construction of $\rho$ amounts to the statement that $P$ is the pullback of this $\bP^1$-bundle by $\beta:Y_n \to H$, i.e.
    \begin{equation}
    \label{equation: pullback of p1-bundle}
        P = Y_n \times_H \mathrm{Bl}_{\Delta_\bot}\bP^{n-1} = \bP({\sh O}_{Y_n} \oplus \beta^*{\sh O}_H(1)).
    \end{equation}	
To justify this, note that $\beta$ and $\gamma$ are compositions of sequences of blowups indexed by $L_n \setminus \{\bot,\top\}$, and, at each step during these mirror processes, the locus blown up in the respective blowup of $\mathrm{Bl}_{\Delta_\bot} \bP^{n-1}$ is the preimage of the locus blown up in the corresponding blowup of $H$, so the claim is clear inductively. Set $\lambda = \rho \circ \nu$. Finally, $\phi_n$ is merely a rational map which will be studied from a modular perspective later, so for now we simply describe it as projection from $[1:\cdots:1]$.

This completes the construction of \eqref{diagram: big diagram} with the exception of the closed immersion $\delta$, which will be constructed below. Note also that, if $\delta$ is removed from the diagram, the diagram is commutative essentially by construction, since all relevant compositions trivially agree at the generic point of the source. 

\begin{rem}\label{rem: regular locus}
    The rational map $\phi_n$ is well defined away from the proper transforms of the diagonals of $\bP^{n-1}$ on $\cA_n$. Indeed, $\alpha$ was obtained by iteratively blowing up (dominant transforms of) proper transforms on $\cA_n$ of polydiagonals in $\bP^{n-1}$, so $\phi_n$ can be identified with $\lambda$ in the complement of the union of the proper transforms on $\cA_n$ of polydiagonals, which coincides with the union of the proper transforms on $\cA_n$ of diagonals in $\bP^{n-1}$. 
\end{rem}

Let $\cC = \overline{M}_{0,n+2} \times_{\overline{M}_{0,n+1}} Y_n$ be the family of curves over $Y_n$ obtained by pulling back the universal family over $\overline{M}_{0,n+1}$, and let $y_1,\ldots,y_{n+1}: Y_n \to \cC$ be the $n+1$ marked sections. 

\begin{lem}\label{lem: description of P^1-bundle}
    We have $P \cong \bP({\sh O}_{Y_n} \oplus y_{n+1}^*\omega_{\cC/Y_n})$ as $\bP^1$-bundles over $Y_n$.
\end{lem}

\begin{proof}
    Let $x_{n+1}$ be the last marked section of $\overline{M}_{0,n+2} \to \overline{M}_{0,n+1}$. By \cite{KapranovChow}*{\S 4.2.3}, we have $\kappa^*c_1({\sh O}_H(1)) = \psi_{n+1} \in \rm{CH}^1(\overline{M}_{0,n+1})$, where $\psi_{n+1}$ is the $\psi$-class corresponding to $x_{n+1}$, so
    \[ 
        c_1(\beta^*{\sh O}_H(1)) = \mu_n^*\kappa^*c_1({\sh O}_H(1)) = \mu_n^*\psi_{n+1} = \mu_n^*c_1(x_{n+1}^* \omega_{\overline{M}_{0,n+2}/\overline{M}_{0,n+1}}) = c_1(y_{n+1}^*\omega_{\cC/Y_n}), 
    \]
    which implies $\beta^*{\sh O}_H(1) \cong y_{n+1}^*\omega_{\cC /Y_n}$ since $c_1:\mathrm{Pic}(Y_n) \to \mathrm{CH}^1(Y_n)$ is an isomorphism, and the conclusion follows from \eqref{equation: pullback of p1-bundle}. 
\end{proof}

	For clarity, we state some well-known facts about compactifications of line bundles and blowups separately, in more natural generality.
	
	\begin{lem}
    \label{lemma: blowing up divisors in a distinguished section}
		Let $X$ be a smooth variety, ${\sh L}$ an invertible ${\sh O}_X$-module, and let $\overline{L} = \bP({\sh O}_X \oplus {\sh L}) \to X$ be the compactification of this line bundle. Let $X_0,X_\infty \subset \overline{L}$ be the two distinguished sections of $\overline{L}$.  
		
		Let $D \subset X$ be a snc divisor on $X$, and $D_0 \subset \overline{L}$ the image of $D$ by $X \cong X_0$. Let $Z$ be the blowup of $\overline{L}$ at the irreducible components of $D_0$ in a given order, and $\pi: Z \to X$ the composite morphism. Then, the following hold:
		\begin{enumerate}
			\item The morphism $\pi$ is flat, and for any $x \in X(\bC)$, $\pi^{-1}(x)$ is a chain of $c_x+1$ rational curves, where $c_x$ is the number of irreducible components of $D$ which contain $x$. 
			\item There exists a (unique) $\bC^*$-action on $Z$ over $X$, such that the birational morphism $Z \to \overline{L}$ is $\bC^*$-equivariant relative to this action and the natural $\bC^*$-action on $\overline{L}$ which fixes $X_0$ and $X_\infty$ pointwise. 
		\end{enumerate}
	\end{lem}
	
	\begin{proof}
		Note that flatness of $\pi$ follows from the second part of (1) by the miracle flatness theorem \cite{stacks-project}*{\href{https://stacks.math.columbia.edu/tag/00R4}{Tag 00R4}}. Let $L = \spec_X \bigoplus_{k \leq 0} {\sh L}^{\otimes k}$ be the line bundle corresponding to ${\sh L}$, so that $X_0$ is the zero section of $L$ and $L= \overline{L} \setminus X_\infty$. 
		
		Let us first consider the case when $D$ is irreducible. The first claim is trivial in this case. The second claim is a standard calculation in local coordinates, or on an affine open cover of $X$ which trivializes ${\sh L}$. Indeed, uniqueness of the $\bC^*$-action ensures gluing. Let $E \subset Z$ be the exceptional divisor, and $F \subset Z$ the other irreducible component of $\pi^{-1}(D)$. Note that
		\[ 
            Z \setminus (X_\infty \cup F) \cong \spec_X \bigoplus_{k \leq 0} {\sh L}(-D)^{\otimes k}. 
        \]
		This allows us to continue inductively, and deduce the claim in general, for any number of irreducible components of $D$. 
	\end{proof}
	
	Let $H' \cong H$ be the exceptional divisor of $\varpi$, so that $H'$ and $H$ are the distinguished sections of $\varpi$. Let $P_0 = \gamma^{-1}(H) \cong Y_n$ and $P_\infty = \gamma^{-1}(H') \cong Y_n$. Recall that, by definition, $\nu$ is the composition of the sequence of blowups of the dominant transforms of $\polyh$, after $\poly$ has been blown up to obtain $P$. However, the dominant transforms of $\polyh$ on $P$ are precisely the images of the boundary divisors on $Y_n$ by $Y_n \cong P_0$, which intersect transversally by \cite{li2009wonderfulcompactificationarrangementsubvarieties}*{Thm. 1.1}. Hence, we are in the situation of \Cref{lemma: blowing up divisors in a distinguished section}, with $Y_n$, $\beta^*{\sh O}_H(1)$, and the boundary of $Y_n$ in the roles of $X$, ${\sh L}$, and $D$ respectively. Hence, by \Cref{lemma: blowing up divisors in a distinguished section},
	\begin{itemize}
		\item $\lambda$ is flat, and for any $q \in Y_n(\bC)$, $\lambda^{-1}(q)$ is a chain of $c_q+1$ rational curves, where $c_q$ is the codimension of the stratum of $Y_n$ which contains $q$; and 
		\item there is a unique $\bC^*$-action on $\cR$ over $Y_n$ compatible with the natural $\bC^*$-action on $P$.
	\end{itemize}
	Therefore, we obtain a morphism
	\begin{equation} Y_n \longrightarrow \mathrm{Chow}_\xi(\cR) \end{equation}
	for suitable $\xi \in \rm{H}_2(\cR,{\bb Z})$, whose image consists of $\bC^*$-invariant $1$-cycles on $\cR$. 

    To complete the diagram \eqref{diagram: big diagram}, $\delta$ is obtained by mapping the proper transform of the section $P_0 \subset P$ on $\cR$, which is still isomorphic to $Y_n$, to $\cA_n$ by $\alpha$. Note that the exceptional locus of $\alpha$ is disjoint from the proper transform of $P_0$ on $\cR$. Indeed, given our construction of $\alpha$, its exceptional locus is a union of dominant transforms of polydiagonals from $\poly$ on $\cR$. However, by \cite{li2009wonderfulcompactificationarrangementsubvarieties}*{Thm. 1.1} the dominant transform on $\cR$ of any polydiagonal in $\poly$ is disjoint from the dominant (or proper) transform of $H$, which coincides with the proper transform of $P_0$ considered above. In particular, $\delta$ is a closed immersion and $Y_n$ is isomorphic to a divisor on $\cA_n$. For future reference, we also record the following fact, which is automatic from the discussion in this paragraph.

    \begin{lem}\label{lem: no blowups near B}
        The morphism $\alpha$ is an isomorphism in a Zariski neighborhood of $\delta(Y_n)$.
    \end{lem}

\section{Moduli of multiscale differentials as wonderful variety}
Throughout this section and the rest of the paper, $\bcg_n$ denotes the moduli space of multiscale differentials $\bP\Xi\overline{M}_{0,n+1}(0^n,-2)$ as introduced in \cite{bcggm}. The aim of this section is to prove:
    \begin{thm}
        \label{multiscale_wonderful}
        The moduli space of multiscale differentials $\bcg_n$ is isomorphic to the wonderful variety $Y_n$.
    \end{thm}
This theorem will be proven by realizing both $Y_n$ and $\bcg_n$ as an explicit sequence of blowups of smooth subschemes of $\overline{M}_{0,n+1}$.

\subsection{The wonderful variety as a blowup of \texorpdfstring{$\overline{M}_{0,n+1}$}{M,0,n+1}}
\label{SS:wonderfulvarietyblowup}
Recall that for any subset $S \subseteq [n]$ such that $2 \leq |S| \leq n-1$, there is a boundary divisor $\delta_S \subset \overline{M}_{0,n+1}$, whose general point is a curve with two $\bP^1$ components intersecting at a point, such that the marked points indexed by $S$ are on one component, and those indexed by $[n+1] \setminus S$ are on the other. Note that $\delta_S$ denotes the geometric boundary divisor not just its class.\footnote{Note also that $\delta_S$ would be denoted by $D^S$ in \cite{keel} if we regard $S$ as a subset of $[n+1]$.} For any partition $\sigma \in L_n \setminus \{ \bot \}$, we have a corresponding closed stratum of $\overline{M}_{0,n+1}$, given by
    \begin{equation}
        \label{equation: comb strata} 
        \overline{M}_\sigma := \bigcap_{S \in B(\sigma), \:|S| \geq 2} \delta_S \subseteq \overline{M}_{0,n+1}, 
    \end{equation}
with the natural convention that $\overline{M}_\top = \overline{M}_{0,n+1}$. A general point of $\overline{M}_\sigma$ is a \emph{comb curve}, with the marked points indexed by $n+1$ and the singleton blocks of $\sigma$ on the backbone, and teeth corresponding to the blocks of $\sigma$ with at least $2$ elements, containing the respective marked points -- see Figure \ref{Fig:combcurve}. Not all boundary strata of $\overline{M}_{0,n+1}$ are of this form, but we will not need to reference the others in what follows.
\begin{figure}[h]
\label{Fig:combcurve}
		\begin{center}
			\begin{tikzpicture}[scale = 1.0]
                \draw [thick] (-2.5,0) -- (2.5,0);
                \draw [thick] (-1.5,.75) -- (-1.5,-2.25);
                \draw [thick] (0,.75) -- (0,-2.25);
                \draw [thick] (1.5,.75) -- (1.5,-2.25);

                \fill[black] (-1.5,-2) circle (3pt);
                \fill[black] (-1.5,-1.25) circle (3pt);
                \fill[black] (-1.5,-.5) circle (3pt);

                \fill[black] (0,-1.25) circle (3pt);
                \fill[black] (0,-2) circle (3pt);

                \fill[black] (1.5,-1.625) circle (3pt);
                \fill[black] (1.5,-.8) circle (3pt);

                \fill[black] (.75,0) circle (3pt);
                \fill[black] (-1.2,0) circle (3pt);
                \fill[black] (-.3,0) circle (3pt);

                \draw node at (-1.9,-2) {$p_1$};
                \draw node at (-1.9,-1.25) {$p_2$};
                \draw node at (-1.9,-.5) {$p_4$};

                \draw node at (-.4,-1.25) {$p_5$};
                \draw node at (-.4,-2) {$p_8$};

                \draw node at (1.9,-1.625) {$p_3$};
                \draw node at (1.9,-.8) {$p_6$};

                \draw node at (.75,.4) {$p_7$};
                \draw node at (-1.2,.4) {$p_9$};
                \draw node at (-.3,.4) {$p_{10}$};
			\end{tikzpicture}
			\caption{A comb curve in $\overline{M}_{0,10}$ corresponding to the partition $\pi = 124|58|36|7|9$ in $L_9$.}
		\end{center}
	\end{figure}
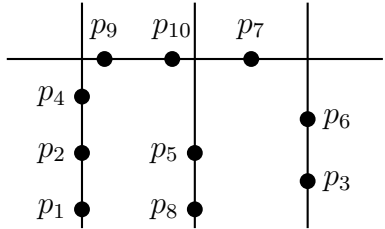

We say that an enumeration of a poset is \emph{increasing} if larger elements come after smaller ones. We partially order the set $L_n$ of partitions of $[n]$ by refinement, that is, $\sigma \leq \sigma'$ if each block of $\sigma'$ is contained in a block of $\sigma$. Let $\sigma_1,\ldots,\sigma_m$ be an increasing enumeration of the partitions of $[n]$ which have at least two blocks with at least two elements.
	
\begin{prop}
\label{proposition: wonderful variety from Mbar}
    There exists a sequence of blowups
    \[ 
        Y_n = X_m \to X_{m-1} \to \cdots \to X_1 \to X_0 = \overline{M}_{0,n+1}, 
    \]
    such that $X_i \to X_{i-1}$ is the blowup of $X_{i-1}$ at the proper transform of $\overline{M}_{\sigma_i}$ by the sequence of blowups $X_{i-1} \to \cdots \to X_0 = \overline{M}_{0,n+1}$, for $i=1,\ldots,m$.
\end{prop}

Before giving the proof of the proposition, we state a general lemma concerning wonderful compactifications, which is implicit in the arguments in \cite{li2009wonderfulcompactificationarrangementsubvarieties}.

\begin{lem}
\label{lem: G-factors commute with blowup}
    Let $Y$ be a smooth projective variety. Let ${\arr G}$ be a building set in $Y$ with induced arrangement ${\arr S}$, $S \in {\arr S}$, and $\mathrm{F}(S) \subset {\arr G}$ the set of ${\arr G}$-factors of $S$ \cite{li2009wonderfulcompactificationarrangementsubvarieties}*{Def. 2.2}. Consider the construction of the wonderful compactification $Y_{\arr G}$ of ${\arr G}$ in \cite{li2009wonderfulcompactificationarrangementsubvarieties}*{\S2} as a sequence of blowups
    \begin{equation}
    \label{equation: Li's increasing blowup construction} 
    Y_{{\arr G}} = Y_N \to Y_{N-1} \to \cdots \to Y_1 \to Y_0 = Y, \end{equation}
    where at each step we blow up a minimal element of the respective building set. 
		
    Then, the dominant transform of $S$ on $Y_{\arr G}$, obtained by successively taking dominant transforms of $S$ in \eqref{equation: Li's increasing blowup construction}, is equal to $\bigcap_{G \in \mathrm{F}(S)} D_G$, where $D_G$ is the boundary divisor of $Y_{\arr G}$ corresponding to $G \in {\arr G}$, cf. \cite{li2009wonderfulcompactificationarrangementsubvarieties}*{Thm. 1.2}.
    \end{lem}
	
	\begin{proof}
		The subvariety $F \in {\arr G}$ such that $Y_1= \mathrm{Bl}_F Y$ is minimal in ${\arr G}$ by assumption. By \cite{li2009wonderfulcompactificationarrangementsubvarieties}*{Prop. 2.8}, ${\arr G}$ gives rise to a building set $\widetilde{\arr G}$ on $Y_1$ consisting of the dominant transforms of the elements of ${\arr G}$. By \emph{loc. cit.}, the arrangement $\widetilde{\arr S}$ induced by $\widetilde{\arr G}$ contains $\widetilde{S}$.
		
		We claim that the $\widetilde{\arr G}$-factors of $\widetilde{S}$ on $Y_1$ are the dominant transforms on $Y_1$ of the ${\arr G}$-factors of $S$ on $Y$. Indeed, this is the special case of the statement in the second line of \cite{li2009wonderfulcompactificationarrangementsubvarieties}*{p. 560} when $k=1$ in the situation of \emph{loc. cit.}, since the set of ${\arr G}$-factors of $S$ is the ${\arr G}$-nest \cite{li2009wonderfulcompactificationarrangementsubvarieties}*{Def. 2.3} induced by the singleton flag $\{S\}$ and, similarly, the set of $\widetilde{\arr G}$-factors of $\widetilde{S}$ is the $\widetilde{\arr G}$-nest induced by the singleton flag $\{\widetilde{S}\}$. However, for the reader's convenience, we provide a more self-contained proof below. Let $\mathrm{F}(S) = \{ G_1,\ldots,G_m \}$. 
		
	   \emph{Case 1.} $S \subseteq F$. Since $F \in {\arr G}$ and $G_1,\ldots,G_m$ are by definition the minimal elements of ${\arr G}$ containing $S$, there exists $k \in [m]$ such that $G_k \subseteq F$. However, $F$ is assumed minimal in ${\arr G}$, so $F = G_k$, that is, $F$ is a ${\arr G}$-factor of $S$. We claim that $\widetilde{S} \subseteq \widetilde{G}_i$ for all $i$. Indeed, if $i=k$, this amounts to the statement that $\widetilde{S}$ is contained in the exceptional divisor, while if $i \neq k$, the transversality of $G_i$ with the center of blowup $G_k$\footnote{This follows from the definition of \emph{building set} in \cite{li2009wonderfulcompactificationarrangementsubvarieties}*{Def. 2.2}. Indeed, one of the stipulations of \emph{loc. cit.} is that the minimal elements of $\cG$ containing $S$, i.e. elements of $\rm{F}(S)$, intersect transversely.}  implies that $\widetilde{G}_i$ coincides with the preimage of $G_i \subset Y$ by $Y_1 \to Y$, but $\widetilde{S}$ is the preimage of $S$ as it is contained in the center of blowup, so $\widetilde{S} \subseteq \widetilde{G}_i$. It remains to check that $\widetilde{G}_1,\ldots,\widetilde{G}_m$ are the minimal elements of $\widetilde{\arr G}$ containing $\widetilde{S}$. Let $G \in {\arr G}$ such that $\widetilde{S} \subseteq \widetilde{G}$. Mapping to $Y$, we obtain $S \subseteq G$, so there exists $i \in [m]$ such that $G_i \subseteq G$. If $i \neq k$, then $\widetilde{G}_i \subseteq \widetilde{G}$ since these dominant transforms are proper transforms. If $i=k$, let $s \in S$ and $E_s$ the fiber of the exceptional divisor $E = \widetilde{G}_k$ over $s$. Then, $E_s \subseteq \widetilde{S} \subseteq \widetilde{G}$. If $G_k \subsetneq G$, then $G$ contains the center of blowup, and its proper transform $\widetilde{G}$ contains a fiber of the exceptional divisor over the center, which is impossible, hence $G = G_k$ and $\widetilde{G} = \widetilde{G}_k$.
		
		\emph{Case 2.} $S \not\subseteq F$. In particular, $G_i \not\subseteq F$ for $i=1,\ldots,m$. Then, the dominant transforms of $S$ and $G_i$ are proper transforms, so $\widetilde{S} \subseteq \widetilde{G}_i$ for all $i$. Let $G \in {\arr G}$ such that $\widetilde{S} \subseteq \widetilde{G}$. Mapping to $Y$, we obtain $S \subseteq G$, so there exists $i$ such that $G_i \subseteq G$. Once again, the dominant transforms of $G_i$ and $G$ are proper transforms, so $\widetilde{G}_i \subseteq \widetilde{G}$, showing that $\widetilde{G}_1,\ldots, \widetilde{G}_m$ are the minimal elements of $\widetilde{\arr G}$ which contain $\widetilde{G}$, thus completing the proof of the claim.
		
		The claim in the second paragraph of this proof and \cite{li2009wonderfulcompactificationarrangementsubvarieties}*{Prop. 2.8} allow us to continue the argument inductively, and conclude that the dominant transform of $S$ on $Y_{\arr G}$ is the intersection of the dominant transforms on $Y_{{\arr G}}$ of the ${\arr G}$-factors of $S$, which are precisely the divisors $D_G$, over all $G \in {\mathrm F}(S)$.
	\end{proof}
    
\begin{proof}[Proof of \Cref{proposition: wonderful variety from Mbar}] 
    Consider an enumeration of $\polyh \setminus \{\Delta^h_\top\}$ as follows: first, an increasing enumeration of $\diagh$, then the enumeration $\Delta_{\sigma_1}^h,\ldots,\Delta_{\sigma_m}^h$ of $\polyh \setminus (\diagh \cup \{\Delta^h_\top\})$. Note that this enumeration of $\polyh \setminus \{\Delta^h_\top\}$ satisfies condition $(*)$ in \cite{li2009wonderfulcompactificationarrangementsubvarieties}*{Thm. 1.3}. Therefore, the final result of the sequence of blowups of (dominant transforms of) polydiagonals in this order is $Y_n$. However, by the discussion of \Cref{Section 2}, we obtain a factorization
    \[ 
        Y_n \xrightarrow{\mu_n} \overline{M}_{0,n+1} \xrightarrow{\kappa} H 
    \]
    of $Y_n \to H$ through Kapranov's construction $\kappa: \overline{M}_{0,n+1} \to H \cong \bP^{n-2}$ of $\overline{M}_{0,n+1}$, as previously observed in \eqref{diagram: big diagram}. The boundary divisor $\delta_S \subset \overline{M}_{0,n+1}$ is the divisor corresponding to the diagonal $\Delta^h_S \subset H$ in the sense of \cite{li2009wonderfulcompactificationarrangementsubvarieties}*{Thm 1.2}. By \Cref{lem: G-factors commute with blowup}, for any $\sigma \in \{\sigma_1,\ldots,\sigma_m\}$, the dominant transform of the polydiagonal $\Delta_\sigma^h$ on $\overline{M}_{0,n+1}$, regarded as the wonderful model of $\diagh$, is $\overline{M}_\sigma$. Then, \cite{li2009wonderfulcompactificationarrangementsubvarieties}*{Thm. 1.3}, completes the proof. (Although we obtain the statement of \Cref{proposition: wonderful variety from Mbar} for dominant transforms of the loci $\overline{M}_{\sigma_i}$, these transforms are in fact proper transforms, since the enumeration $\sigma_1,\ldots,\sigma_m$ is increasing.)
    \end{proof}
	
    On the other hand, precisely the same construction produces $\bcg_n$ by \cite{Devkotacohomology}, as we will show in the next subsection.

\subsection{Moduli of multiscale differentials as a blowup of \texorpdfstring{$\overline{M}_{0,n+1}$}{M0,n+1}}

Before we state the precise analog of \Cref{proposition: wonderful variety from Mbar} for $\bcg_n$, let us review the structure of its boundary divisor as described in \cite{bcggm} and \cite{Devkotacohomology}.

The boundary of $\bcg_n$ admits a stratification with strata indexed by stable rooted level trees $\Lambda$, where the $(-2)$-order marked point lies on the top level, which consists of a single root vertex. Given such a level tree $\Lambda$, let us denote by $D_\Lambda$ the corresponding stratum. The number of levels below the root (equivalently the number of level passages) equals the codimension of $D_\Lambda$ in $\bcg_n$. In particular, the irreducible components of the boundary divisor of $\bcg_n$ are of the form $D_\Gamma$ for $\Gamma$ a rooted level tree with two levels; if, in addition, the lower level of $\Gamma$ has two or more vertices then $D_\Gamma$ is exceptional over $\overline{M}_{0,n+1}$.

The boundary of $\bcg_n$ regarded as compactification of $M_{0,n+1}$ is a simple normal crossings divisor; as such, the intersection $D_{\Gamma_1}\cap D_{\Gamma_2}$ of two boundary divisors is the irreducible stratum corresponding to a stable rooted level tree $\Lambda$ with three levels (i.e. two level passages) determined by $\Gamma_1$ and $\Gamma_2$. Conversely, $\Gamma_i$ is obtained from $\Lambda$ by contracting one of the two level passages for $i=1,2$. With this, we introduce a partial ordering on the boundary divisors of $\bcg_n$, exceptional over $\overline{M}_{0,n+1}$. We say $D_{\Gamma_1}< D_{\Gamma_2}$ if $D_{\Gamma_1}\cap D_{\Gamma_2}\eqqcolon D_\Lambda$ is non-empty and $\Gamma_1$ is obtained from $\Lambda$ by contracting the bottom level passage, and $\Gamma_2$ by contracting the top level passage. 

Equivalently, $D_{\Gamma_1}< D_{\Gamma_2}$ if $D_{\Gamma_1}\cap D_{\Gamma_2}$ is non-empty and the generalized stratum given by the bottom level of $D_{\Gamma_1}$ has higher dimension than that given by the bottom level of $D_{\Gamma_2}$.\footnote{If $\Gamma$ is a level graph with two levels and $k$ vertices in the bottom level then the dimension of the generalized stratum corresponding to its lower level is equal to the sum of the dimensions of the moduli spaces $M_{0,i_j}$ corresponding to each vertex in the lower level, plus $k-1$ coming from the differentials on each of the $k$ vertices up to $\bC^*$-action -- see \cite{cmz22}*{\S4} for the precise definition of generalized stratum.} Then for any total order, also denoted ``$<$'', refining this partial order, $\tau:\bcg_n\rightarrow\overline{M}_{0,n+1}$ factorizes into a sequence of blowups along smooth centers $X_m=\bcg_n\rightarrow X_{m-1}\rightarrow\cdots\rightarrow X_0=\overline{M}_{0,n+1}$ by blowing down the exceptional divisors in descending order -- see \cite{Devkotacohomology}*{\S 5}. Because $\tau(D_{\Gamma_1})\cap \tau(D_{\Gamma_2})=\emptyset$ when $D_{\Gamma_1}\cap D_{\Gamma_2}=\emptyset$ for two divisors $D_{\Gamma_1}$ and $D_{\Gamma_2}$, the choice of refinement of the partial order is irrelevant. This sequence of blowups will furnish us the analog of \Cref{proposition: wonderful variety from Mbar} for $\bcg_n$.

The proof of \Cref{proposition: Prabhat's construction2} below boils down mostly to some simple combinatorics. In the context of $\bcg_n$, a stable rooted level tree $\Gamma$ is equipped with $n+1$ half-edges which encode the irreducible components on which the corresponding marked points lie. As mentioned above, the $(-2)$-order marked point lies on the root, but the other marked points $p_1,\ldots, p_n$ can lie on any component. The stability condition enforces that each irreducible component has at least three special points -- that is, marked points or nodes. Given a two-level stable rooted level tree $\Gamma$, we define a partition $\sigma(\Gamma) \in L_n$ given by $i\sim_{\sigma(\Gamma)} j$ if and only if $p_i$ and $p_j$ lie on the same non-root component of $\Gamma$. This defines a map
\[
    \sigma:\left\{\parbox{3.4cm}{\centering two-level stable marked level trees}\right\} \longrightarrow L_n \setminus \{\top,\bot\}.
\]
The non-root vertices of $\Gamma$ correspond bijectively to the non-singleton blocks of $\sigma(\Gamma)$ and the points attached to the root of $\Gamma$ correspond to the singleton blocks of $\sigma(\Gamma)$. Consequently, $\sigma$ is invertible. See Figure \ref{Fig:combcurve} for a picture.\footnote{Note that in that example $n = 9$, because we regard the ($-2$)-order marked point as being $p_{10}$ and thus not contributing to the data of the partition.} Note that $\sigma$ maps the set of dual trees of exceptional divisors for $\bcg_n \to \overline{M}_{0,n+1}$ to the set of partitions having at least two non-singleton blocks, which we denote by $\rm{Ex}_n$.

\begin{prop}
    \label{proposition: Prabhat's construction2}
	There exists a sequence of blowups
		\[ 
            \bcg_n = X_m \to X_{m-1} \to \cdots \to X_1 \to X_0 = \overline{M}_{0,n+1} 
        \]
	such that $X_i \to X_{i-1}$ is the blowup of $X_{i-1}$ at the proper transform of $\overline{M}_{\sigma_i}$ by the sequence of blowups $X_{i-1} \to \cdots \to X_0 = \overline{M}_{0,n+1}$, for $i=1,\ldots,m$.
\end{prop}

\begin{proof}
    We need to verify that in the sequence of blowups factoring $\tau:\bcg_n\rightarrow \overline{M}_{0,n+1}$ constructed in \cite{Devkotacohomology}*{\S 5}, each step blows up the dominant transform of $\overline{M}_{\sigma}$ for some $\sigma \in \rm{Ex}_n$ in order of an increasing enumeration of $\rm{Ex}_n\subset L_n$. For any $\tau$-exceptional divisor $D_\Gamma$ one has 
        \[
            \tau(D_\Gamma) \:\:= \bigcap_{b\in B(\sigma(\Gamma)),\lvert b\rvert \ge 2} \delta_b \:\: \subseteq \:\: \overline{M}_{0,n+1}
        \]
    in the notation of \Cref{SS:wonderfulvarietyblowup}; that is, $\tau(D_\Gamma) = \overline{M}_{\sigma(\Gamma)}$ with $\sigma(\Gamma) \in \rm{Ex}_n$. On the other hand, $\{\Delta_\sigma^h:\sigma \in \rm{Ex}_n\} = (\polyh \setminus \{\Delta_\top^h\})\setminus \diagh$ and it is exactly the strict transforms $\overline{M}_\sigma$ of $\Delta_\sigma^h$ for each $\sigma \in \rm{Ex}_n$ that are blown up to construct $Y_n$ from $\overline{M}_{0,n+1}$. Finally, if $D_{\Gamma_1} \cap D_{\Gamma_2}\ne \varnothing$, the condition $D_{\Gamma_1}<D_{\Gamma_2}$ is equivalent to the statement that $\sigma(\Gamma_1)$ is refined by $\sigma(\Gamma_2)$. This means that the ordering on the divisors $D_\Gamma$ in $\bcg_n$ gives an increasing ordering of $\rm{Ex}_n$, as desired. 
	\end{proof}

\Cref{multiscale_wonderful} now follows immediately from \Cref{proposition: wonderful variety from Mbar} and \Cref{proposition: Prabhat's construction2}. 

\begin{cor}
\label{C: special divisor2}
    $\bcg_n$ embeds into $\cA_n$ as the boundary divisor corresponding to the partition $\top$. 
\end{cor}

\begin{proof}
    The embedding is the composition $\bcg_n \cong Y_n \hookrightarrow \cA_n$, by \Cref{multiscale_wonderful} and \eqref{diagram: big diagram}. 
\end{proof}

As stated in \Cref{C: special divisor2}, the boundary divisor in $\cA_n$ represented by the rooted level tree in \Cref{Fig: special level tree in A_n2} is isomorphic to $\bcg_n$.
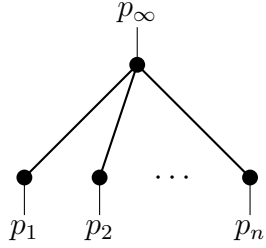
\begin{figure}[H]
\centering
    \begin{tikzpicture}
        \fill[black] (0,0) circle (3pt);
        \fill[black] (-1.5,-1.5) circle (3pt);
        \fill[black] (-.5,-1.5) circle (3pt);
        \fill[black] (1.5,-1.5) circle (3pt);
        
        \draw[thick] (0,0) -- (-1.5,-1.5);
        \draw[thick] (0,0) -- (-.5,-1.5);
        \draw[thick] (0,0) -- (1.5,-1.5);
        \draw (0,0) -- (0,.5);
        \draw(-1.5,-1.5) -- (-1.5,-2);
        \draw(-.5,-1.5) -- (-.5,-2);
        \draw(1.5,-1.5) -- (1.5,-2);

        \draw node at (.5,-1.5) {$\cdots$};
        \draw node at (-1.5,-2.2) {$p_1$};
        \draw node at (-.5,-2.2) {$p_2$};
        \draw node at (1.5,-2.2) {$p_n$};
        \draw node at (0,.7) {$p_\infty$};
    \end{tikzpicture}
\caption{Dual tree to boundary divisor in $\cA_n$ isomorphic to $\bcg_n$}
\label{Fig: special level tree in A_n2}
\end{figure}

Geometrically, the identification of the boundary divisor in $\cA_n$ with $\bcg_n$ amounts to stabilization of the underlying nodal curve of the multiscaled line -- the multiscale differential on the resulting stable curve is then given by the collection of differentials on the stable components of the multiscaled line. Furthermore, it follows from the proof of \Cref{proposition: Prabhat's construction2} that the divisor in $\bcg_n$ corresponding to a two-level tree $\Gamma$ is sent to the codimension $2$ stratum of $\cA_n$ corresponding to the chain $\sigma(\Gamma) < \top$ in $L_n$ in the notation of \cite{robotis2024spaces}*{\S 3}.

\begin{rem}
    The construction of $\cA_n$ is modular in nature, with the closed points of $\cA_n$ being defined as complex projective isomorphism classes of multiscale lines \cite{augmentedstability}*{\S 4}; nevertheless, there is no obvious moduli functor which $\cA_n$ represents. Similarly, though \cite{bcggm} proves the existence of a moduli functor which $\cB_n$ represents locally, there does not seem to be a complete global description. By contrast, the variety $\overline{P}_n$ introduced in \cite{Zahariucmarkednodalcurvesvfs} represents a moduli functor $F_n$ such that $\overline{P}_n(\bC)$ is a set of marked genus $0$ curves with logarithmic vector fields, which, as explained in \cite{robotis2024spaces}*{\S 5.1}, give rise to differentials on terminal components. Consequently, there is a corresponding universal family $\cU_n \to \overline{P}_n$ of marked $n+1$ marked genus $0$ curves with differentials. 

    On the other hand, there is a resolution morphism $\xi:\cA_n \to \overline{P}_n$ described in \cites{robotis2024spaces,zahariuc2024smallresolutionsmodulispaces} which has the effect of forgetting the level structure and non-terminal differentials of points of $\cA_n$. Based on these considerations, it seems reasonable to expect that there is a natural moduli functor $E_n$ which $\cA_n$ represents such that:
    \begin{enumerate}
        \item $\xi:\cA_n \to \overline{P}_n$ corresponds to a forgetful functor $E_n \to F_n$ for each $n$; and 
        \item $\cB_n$ represents a closed subfunctor of $E_n$, corresponding to the closed immersion $\cB_n \hookrightarrow \cA_n$ of \Cref{C: special divisor2}.
    \end{enumerate}
\end{rem}

We are now also able to prove \Cref{prop: resolution of A to B map}. 

\begin{proof}[Proof of \Cref{prop: resolution of A to B map}]
    Since $\bcg_n = Y_n$ by \Cref{multiscale_wonderful}, the diagram in the statement of the proposition is part of diagram \eqref{diagram: big diagram}, and commutativity was shown in \Cref{Section 2}. The description of the $\bP^1$-bundle $P$ is Lemma \ref{lem: description of P^1-bundle}. The fact that $\alpha$ is a composition of blowups with smooth centers disjoint from $\bcg_n \hookrightarrow \cA_n$ follows from the construction in \Cref{Section 2} and \Cref{lem: no blowups near B}. 
\end{proof}

Our next corollary involves the Chow ring $\underline{\mathrm{CH}}^*(M)$ of the graphic matroid $M=M(K_n)$. We refer to \cite{bhmpw2022semismalldecompositionofchowringofmatroid} for the exact definitions of matroids and their Chow rings. 

\begin{cor}
\label{C:chowring2}
    The Chow ring $\mathrm{CH}^*(\bcg_n)$ of $\bcg_n$ is isomorphic to the Chow ring $\underline{\mathrm{CH}}^*(M)$ of the graphic matroid $M=M(K_n)$.
\end{cor}

\begin{proof}
    This follows from \cite{bhmpw2022semismalldecompositionofchowringofmatroid}*{Rem. 2.13} or \cite{feichtneryuzvinskychowring}*{Cor. 2}.
\end{proof}

In particular, $\rm{CH}^*(\bcg_n)$ is isomorphic to the ring $\bZ[x_\Gamma]/(I+J)$, where there is one variable $x_\Gamma$ for each two-level stable marked level trees corresponding to a divisor of $\bcg_n$ and
\begin{enumerate}
    \item $I$ is an ideal of linear forms:
    \[
        \sum_{i\sim_{\sigma(\Gamma)}j} x_\Gamma \:\: - \:\:\sum_{k\sim_{\sigma(\Gamma)}\ell} x_\Gamma 
    \]
    for all $i,j,k,\ell$ in $[n]$ such that $i \neq j$ and $k \neq \ell$; and \vspace{2mm}
    \item $J$ is an ideal of quadratic forms: $x_\Gamma\cdot x_{\Gamma'}$ if $\sigma(\Gamma)$ and $\sigma(\Gamma')$ are not comparable in $L_n$.
\end{enumerate}

Consequently, $\mathrm{CH}^*(\bcg_n)$ is generated by the boundary divisors $D_\Gamma$, with relations generated in degrees one and two. The degree 1 relations in (1) are linear WDVV relations, while the quadratic relations in (2) reflect the geometrically obvious relation of $D_{\Gamma_1}\cdot D_{\Gamma_2}=0$ when the divisors $D_{\Gamma_1}$ and $D_{\Gamma_2}$ are disjoint. This refines the result in \cite{Devkotacohomology} and draws a parallel to the Chow ring structure of $\overline{M}_{0,n+1}$ -- see \cite{keel} or \cite{acg2}*{\S 17.7}.
\section{\texorpdfstring{$\bC^*$}{C*}-action on \texorpdfstring{$\cA_n$}{An}}\label{Section 4}

There is a natural $\bC^*$-action on $\cA_n$, which we now explain -- see \Cref{S:introduction} for some of the relevant notation. Recall that for any $\Sigma \in \cA_n$, the irreducible components furthest from the root are equipped with non-zero meromorphic differentials with a unique pole of order $2$ at the node connecting them to the higher levels of the curve. We define 
\[ 
    \mu:\bC^* \times \cA_n \to \cA_n \quad \text{by} \quad \mu(\lambda,\Sigma) = \Sigma^\lambda
\]
where $\Sigma^\lambda$ is the multiscale line obtained from $\Sigma$ by scaling all bottom level differentials by $\lambda$. It is clear from the definition that $\mu$ preserves the combinatorial type of $\Sigma$; i.e. we have an equality of the corresponding level trees $\Gamma(\Sigma) = \Gamma(\Sigma^\lambda)$ for all $\lambda \in \bC^*$.

\begin{lem}
\label{L:action}
    The map $\mu$ defines a group action in the category of varieties over $\bC$. Furthermore, the action is generically free and stabilizer groups are either trivial or $\bC^*$.
\end{lem}

\begin{proof}
    The reader can readily verify that $\mu$ is an action in the category of sets. To verify algebraicity, we use local coordinates. Fixing a level tree $\Gamma$, we obtain an open coordinate neighborhood $U_\Gamma$ containing the stratum $S_\Gamma$ -- see \cite{augmentedstability}*{\S 4.1}. There are two types of coordinate functions on $U_\Gamma$: $t_1,\ldots, t_\ell$, which control the formation of new levels in the tree, and $\{z_{ij}\}$ which record distances between nodes on the same irreducible component of $\Sigma \in U_\Gamma$. 
    
    A calculation shows that the coordinates transform as $t_i(\Sigma^\lambda) = \lambda^{-\delta_{i\ell}} t_i(\Sigma)$ where $\delta_{i\ell}$ is the Kronecker delta and $z_{ij}(\Sigma^\lambda) = \lambda^{\epsilon_{ij}} z_{ij}(\Sigma)$, where $\epsilon_{ij}$ is $1$ if $p_i$ and $p_j$ are on the same irreducible component and $0$ otherwise. Next, $\Sigma$ has nontrivial stabilizer group if and only if for each terminal component $\Sigma_\tau \subset \Sigma$, we have $p_i,p_j \in \Sigma_\tau$ implies that $p_i = p_j$. Indeed, otherwise there exist $p_i\ne p_j$ on some $\Sigma_\tau$ so that $\Pi_{ij}(\Sigma^\lambda) = \lambda\cdot \Pi_{ij}(\Sigma) \ne \Pi_{ij}(\Sigma)$, where here $\Pi_{ij}(\Sigma) = \int_{p_i}^{p_j} \omega_\tau.$ This reasoning shows that the locus of points with trivial stabilizer is Zariski open and nonempty and if $\Sigma$ has a nontrivial stabilizer, then it is $\bC^*$.
\end{proof}

\begin{rem}
    Note that the functions $\Pi_{ij}:\cA_n\to \bP^1$ are $\bC^*$-equivariant with respect to this action on $\cA_n$ and the usual action on $\bP^1$.
\end{rem}

\begin{ex}[$n=3$]
\label{Ex:n=3}
    $\cA_3$ is the surface obtained from $\bP^2$ by blowing up three points on a line $\ell$. We compute the fixed loci for the $\bC^*$-action: in the locus of irreducible curves $\cA_3^\circ$, the only fixed point is $q = (\bP^1,\infty, dz,p_1,p_2,p_3)$ where $p_1 = p_2 = p_3\ne \infty$. In the boundary $\cA_3\setminus\cA_3^\circ$ there are three isolated fixed points $q_{ij}$ corresponding to $p_i = p_j$ for $\{i<j\}\subset \{1,2,3\}$. Furthermore, the divisor $D$ obtained as the strict transform of $\ell$ is fixed. All chains of orbit closures contain $q$, so it is impossible to obtain a reasonable quotient for this action directly.

    \begin{figure}[H]
		\centering
		\begin{tikzpicture}[scale =1.3]
			\draw [ultra thick] (-1.5,0) to (1.5,0) node [anchor=west] {$D \cong \bcg_3$};
			\draw (-1,-0.5) -- (-1,1);
			\draw (0,1) -- (0,-0.5) node [anchor = north] {$\cA_3$};
			\draw (1,-0.5) -- (1,1);
			
			\draw (-0.5,2) -- (-1.1,0.4);
			\draw (-0.5,2) -- (0.1,0.4);
			\draw (-0.5,2) -- (1.3,0.4);
			
			\fill (-1,0.67) circle (1.5pt) node [anchor = west] {$q_{12}$};
			\fill (0,0.67) circle (1.5pt) node [anchor = west] {$q_{13}$};
			\fill (1,0.67) circle (1.5pt) node [anchor = west] {$q_{23}$};
			\fill (-0.5,2) circle (1.5pt) node [anchor = west] {$q$};
			
			\draw [dotted] (-0.8,2) to (-3.2,2);
			\fill (-3.5,2) circle (2pt); \draw (-3.5,2) -- (-3.5,1.5) node [anchor=north] {$p_1=p_2=p_3$}; \draw (-3.5,2) -- (-3.5,2.3);
			\draw [dotted] (-1.2,0.67) to (-3.2,0.67);
			\fill (-3.5,0.67) circle (2pt);  \fill (-3,0.17) circle (2pt); \fill (-4,0.17) circle (2pt);  \draw (-3,0.17) -- (-3.5,0.67) -- (-4,0.17); \draw (-3.5,0.67) -- (-3.5,0.97);
			\draw (-3,0.17) -- (-3,0.17-0.5) node [anchor=north] {$p_3$}; \draw (-4,0.17) -- (-4,0.17-0.5) node [anchor=north] {$p_1=p_2$}; 	
			\fill (4,1.5) circle (2pt);  \fill (3.5,1) circle (2pt); \fill (4,1) circle (2pt);  \fill (4.5,1) circle (2pt); 
			\draw (4,1.5) -- (3.5,1) -- (3.5,0.5) node [anchor=north] {$p_1$};	
			\draw (4,1.5) -- (4,1) -- (4,0.5) node [anchor=north] {$p_2$};	
			\draw (4,1.5) -- (4.5,1) -- (4.5,0.5) node [anchor=north] {$p_3$};	
                \draw (4,1.5) -- (4,1.8);
			\draw [dotted] (3.3,1) -- (1.2,0.05);
		\end{tikzpicture}
	\caption{The fixed locus of $\mu$ for $n=3$ is $D \cup \{q,q_{12},q_{13},q_{23}\}$.}
	\label{Fig: fixed locus for n=3}
	\end{figure}
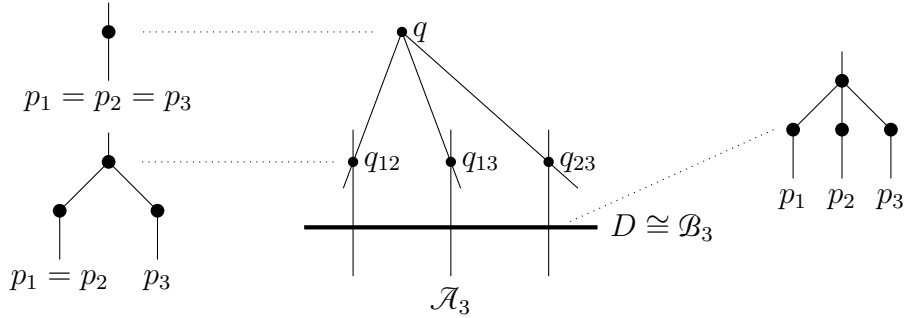
    
    We consider next the blowup $X = \Bl_q(\cA_3)$. $\bC^*$ acts trivially on $\bP(T_q\cA_3)$ so $\mu$ induces an action on $X$ which is trivial on the exceptional divisor. All orbit closures with respect to the induced $\bC^*$-action are now disjoint and the quotient by $\bC^*$ gives a map $X\to D$ with generic fiber $\bP^1$, where $D \cong \bP^1$ is the strict transform of $\ell$. There are three distinguished points $x_1,x_2,x_3$ of $D$, corresponding to the points of intersection with the three exceptional divisors in $\cA_3$. The fibers of $X\to D$ are $\bP^1$ for all $x\in D\setminus \{x_1,x_2,x_3\}$ and degenerate to $A_2$ fibers over each $x_i$. In particular, $D$ is equipped with a flat family of rational cycles in $\cA_3$ coming from the $\bC^*$-action.
\end{ex}

\subsection{\texorpdfstring{$\bcg_n$}{Bn} as a Chow quotient of \texorpdfstring{$\cA_n$}{An}}
Our next objective is to realize $\bcg_n$ as the normalized \emph{Chow quotient} of $\cA_n$ by the $\bC^*$-action described above. Chow quotients were introduced by Kapranov as a method of forming quotients of projective varieties by actions of algebraic groups without recourse to GIT -- see \cite{KapranovChow}. 

To define the Chow quotient, consider the action of an algebraic group $G$ on a projective variety $X$. One can find a $G$-invariant open subset $U\subset X$ such that for all $x\in U$ the homology class $\beta$ of the orbit closure of $x$ is independent of $x$. Then, sending $x$ to its orbit closure gives an injection $U/G\hookrightarrow \Chow_\beta(X)$ into the Chow variety of cycles of class $\beta$ in $X$. Since $\Chow_\beta(X)$ is proper (in fact, projective), we define the Chow quotient $X\sslash G$ as the closure of $U/G$ in $\Chow_\beta(X)$. The Chow quotient carries a tautological flat family of cycles obtained from pulling back the universal family $\cU_\beta \to \Chow_\beta(X)$ -- see \cite{kollarrationalcurves}*{\S I Thm. 3.21}.

In the special case where $G = \bC^*$, the theory takes on a rather explicit form and is developed in the recent work \cite{ChowquotientCstar}. There, the authors write $\cC(X)$ for the normalisation of the Chow quotient of a variety $X$ by $\bC^*$. We establish some facts about the $\bC^*$-action on $\cA_n$. 

By the discussion in \Cref{Section 2}, the map $\lambda:\cR \to \bcg_n$ is flat with generic fiber $\bP^1$ and with all other fibers singular rational curves of $A_k$-type. Furthermore, for each $p \in \bcg_n$ the fiber $\lambda^{-1}(p) = \cR_p$ has homology class $\xi \in \rm{H}_2(\cR,\bZ)$ which is independent of $p$. This gives an injection $\bcg_n \to\Chow_\xi(\cR)$. Now, the proper morphism $\alpha:\cR\to \cA_n$ induces a morphism $\alpha_*:\Chow_\xi(\cR)\to \Chow_{\alpha_*\xi}(\cA_n)$ which sends $\cR_p$ to $\alpha(\cR_p)$ by \cite{barletcycles}*{Lem. 4.1.11}. Denote the composite morphism $\chi:\bcg_n \to \Chow_{\alpha_*\xi}(\cA_n)$. Also, denote by $\bcg_n^\circ$ the open dense subset of $\bcg_n$ consisting of irreducible curves. When embedded in $\cA_n$, this corresponds to the locus of curves having dual tree exactly that pictured in \Cref{Fig: special level tree in A_n2}.

\begin{thm}\label{thm: normalized Chow quotient}
    The morphism $\chi$ induces an isomorphism $\bcg_n\to \cC(\cA_n)$, where $\cC(\cA_n)$ is the normalized Chow quotient of $\cA_n$ by $\bC^*$.
\end{thm}

\begin{proof}
    We first verify that $\chi:\bcg_n \to \Chow_{\alpha_*\xi}(\cA_n)$ is injective on closed points. For this, recall from \eqref{diagram: big diagram} that there is a commutative diagram 
    \[
    \begin{tikzcd}
        \cR\arrow[dr,"\lambda"]\arrow[d,"\alpha",swap]&\\
        \cA_n\arrow[r,dashed,"\phi_n"]&\bcg_n
    \end{tikzcd}
    \]
    where the rational map $\phi_n$ is defined (at least) on the open set $U \subset \cA_n$ where $\Pi_{ij} \ne 0$ for all $1\le i < j \le n$; that is, where $p_i\ne p_j$ for all $i\ne j$ -- see \Cref{rem: regular locus}. Note that the image of the section $\delta$ is contained in $U$ and so $\phi_n$ maps surjectively to $\bcg_n$. Thus, $\alpha^{-1}(U)\cap \lambda^{-1}(p)\ne \varnothing$ for all $p\in \bcg_n$ and $\phi_n(\alpha(\cR_p)) = p$.

    In \cite{augmentedstability}*{Lem. 4.4}, a collection of algebraic coordinate patches is constructed for $\cA_n$. In particular, there is a coordinate system around $\bcg_n^\circ$, $(U,t,z_{13},\ldots, z_{1n})$, where $t = \Pi_{12}^{-1}$ and $z_{1j} = \Pi_{1j}/\Pi_{12}$ and $U$ is the $\bC^*$-invariant open set containing $\bcg_n^\circ$ and all elements of $\cA_n^\circ$ such that $\Pi_{ij}\ne 0$ for all $1\le i <j\le n$. From this description it follows that $\bC^*$ acts with weight $0$ on the coordinates $z_{13},\ldots, z_{1n}$ and weight $-1$ on $t$. The coordinates realize $U$ as the complement of a $\bC^*$-invariant hyperplane arrangement in $\Spec \bC[t,z_{13},\ldots, z_{1n}]$ and thus $U/\bC^*$ is identified with its image in $\Spec \bC[t,z_{13},\ldots, z_{1n}]^{\bC^*} = \Spec \bC[z_{13},\ldots, z_{1n}]$ which is $\bcg_n^\circ$.

    Next, \Cref{lemma: blowing up divisors in a distinguished section} implies that for any $p\in \bcg_n^\circ$ one has $\lambda^{-1}(p) \cong \bP^1$ and the description of $U$ implies that $\phi_n^{-1}(p) \cap U \cong \bA^1$. Since the arrangement blown up to form $\cR$ from $\cA_n$ is contained in the complement of $U$, $\alpha$ maps the component of $\lambda^{-1}(p)$ containing $p\in \bcg_n \subset \cR$ bijectively onto the component of $\phi_n^{-1}(p)$ containing $p\in \bcg_n \subset \cA_n$. Consequently, for all $x\in U$ the Chow class of the orbit closure of $x$ is constant since it is of the form $\alpha(\cR_p)$. So, we obtain a commutative diagram:
    \[
        \begin{tikzcd}
            U/\bC^*\arrow[r]&\Chow_{\alpha_*\xi}(\cA_n)\\
            \bcg_n^\circ\arrow[r,hook]\arrow[u,"\sim"]&\bcg_n\arrow[u,"\chi",swap]
        \end{tikzcd}
    \]
    which implies that $\chi$ maps bijectively on closed points onto $\cA_n\sslash \bC^*$. Since $\cB_n^\circ = U\cap \bcg_n \to U/\bC^*$ is an isomorphism, it follows that $\chi$ is birational onto its image. Therefore, by the universal property of normalization and Zariski's main theorem \cite{stacks-project}*{\href{https://stacks.math.columbia.edu/tag/0AB1}{Tag 0AB1}}, there is an induced isomorphism $\chi:\bcg_n \to \cC(\cA_n)$.
\end{proof}

\begin{rem}
    There is another way to construct the morphism $\cB_n \to \Chow_{\alpha_*\xi}(\cA_n)$. Indeed, one can construct an explicit family of cycles $\cZ \subseteq \cA_n \times \cB_n$ over $\cB_n$ whose fiber over $p\in \cB_n$ is exactly the chain of orbit closures of $\mathbb{C}^*$-orbits containing $p$. This approach yields the same morphism since they agree set theoretically over $\cB_n^\circ$, but the verification that $\cZ$ is a family of algebraic cycles requires a bit more effort. It is, however, somewhat striking that $\cZ$ is defined by the ``determinantal'' relations 
    \[
        \left\{ \frac{\Pi_{ij}}{\Pi_{k\ell}} = \frac{\Pi_{ij}|_{\cB_n}}{{\Pi_{k\ell}|_{\cB_n}}}\right\}_{i,j,k,\ell}.
    \]
    Note that while $\Pi_{ij}|_{\cB_n}$ is identically $\infty \in \bP^1$, the ratio $\Pi_{ij}|_{\cB_n}/\Pi_{k\ell}|_{\cB_n}$ need not be. 
\end{rem}

\begin{rem}
    Note that since $\lambda:\cR\rightarrow\bcg_n$ is flat, the fibers have same Hilbert polynomial. So, using the same steps as above, we also obtain an injection into the Hilbert scheme $\bcg_n\rightarrow \mathcal{Hilb}(\cA_n)$, and thus the normalization of the Hilbert quotient of $\cA_n$ under $\bC^*$-action is also isomorphic to $\bcg_n$ -- see \cite{KapranovChow}*{\S 0.5} for the precise definition of Hilbert quotient. In fact, for an equalized $\bC^*$-action, the normalizations of the Chow quotient and the Hilbert quotient are always isomorphic -- see \cite{ChowquotientCstar}*{Rem. 4.6}.
\end{rem}

\subsection{\texorpdfstring{$\bcg_n$}{Bn} as a GIT quotient of \texorpdfstring{$\cA_n$}{An}} 

In this subsection, we will use the description of the $\bC^*$-action on $\cA_n$ as a natural lift of a $\bC^*$-action on $\bP^{n-1}$ via the blowup description of $\cA_n$ to realize $\bcg_n$ as a GIT quotient of $\cA_n$. There is a $\bC^*$-action on $\bP^{n-1}$ that fixes the point $[1:\ldots:1]$ and the hyperplane $H=\{x_1+\ldots+x_n=0\}$, and acts freely otherwise. We consider an $\mathcal{O}_{\bP^{n-1}}(1)$-linearized $\bC^*$-action $\xi$ with this property given by the linear representation of $\bC^*$ on the vector space $H^0(\bP^{n-1},\mathcal{O}_{\bP^{n-1}}(1))$ defined by
\begin{align*}
  t\cdot (x_1+\ldots+x_n) &= t^{-1}(x_1+\ldots+x_n)\qquad \text{and}\\
  t\cdot (x_1-x_j) &= x_1-x_j \qquad \text{for } j>1.
\end{align*}
Then, as computed in \cite{dolgachevgit}*{Ex. 8.2}, the semistable locus for this action is $(\bP^{n-1})^{\mathrm{ss}}=\bP^{n-1}\setminus\{[1:\ldots:1]\}$, while the stable locus is empty, since a generic orbit-closure will always intersect the fixed locus $H$. The $\bC^*$-invariant graded ring of sections is:
\[
    \bigoplus_{k\geq 0}H^0(\bP^{n-1},\mathcal{O}_{\bP^{n-1}}(k))^{\bC^*}\cong \bC[x_1-x_2,x_1-x_3,\ldots,x_1-x_n].
\]
Consequently, the resulting GIT quotient $\bP^{n-1}\sslash_\xi \bC^*$ is isomorphic to $\bP^{n-2}$. The quotient map $(\bP^{n-1})^{\rm{ss}} \to \bP^{n-1}\sslash_\xi \bC^*$ restricts to an isomorphism $H\xrightarrow{\sim} \bP^{n-2}$. 

\begin{thm}\label{thm: GIT quotient}
  There is an ample line bundle $\mathcal{L}$ on $\cA_n$ that admits a $\bC^*$-linearization such that the resulting GIT quotient $\cA_n\sslash_{\mathcal{L}}\bC^*$ is isomorphic to $\bcg_n$.
\end{thm}

\begin{proof}
    The idea is to iteratively apply \cite{kirwanblowup}*{Lem. 3.11} to the sequence of birational morphisms $\cA_n=X_m\rightarrow X_{m-1}\rightarrow\ldots\rightarrow X_0=\bP^{n-1}$ obtained by blowing up polydiagonals in $\mathcal{Poly}_\infty$ in the increasing order described in \Cref{Section 2}. For the purpose of induction, let us consider the blowup $f_k:X_k\rightarrow X_{k-1}$ along the proper transform $V_{k-1}\subset X_{k-1}$ of $\Delta_{\sigma_k}^h$. Denote by $H_{k-1}$ the proper transform of $H$ in $X_{k-1}$ and by $E_k\subset X_k$ the exceptional divisor of $f_k$; $H_{k-1}$ will still be contained in the fixed locus under the induced $\bC^*$-action on $X_{k-1}$ and $V_{k-1}$ is contained in $H_{k-1}$. We assume that there is a $\bC^*$-linearized ample line bundle $\mathcal{L}_{k-1}$ on $X_{k-1}$ such that the corresponding GIT quotient $X_{k-1}\sslash_{\mathcal{L}_{k-1}}\bC^*$ is isomorphic to $H_{k-1}$. When $k=1$, the ample line bundle is $\mathcal{O}_{\bP^{n-1}}(1)$ with the aforementioned linearization $\xi$; this also furnishes the base case of the induction.

    Then for $k>1$, \cite{kirwanblowup}*{Lem. 3.11} implies that for appropriately chosen integer $d_k$, the line bundle $\mathcal{L}_{d_k}\coloneqq f_k^*\mathcal{L}_{k-1}^{\otimes d_k}\otimes\mathcal{O}_{X_k}(-E_k)$ is very ample, and the induced $\bC^*$-linearization gives a GIT quotient $X_k\sslash_{\mathcal{L}_{d_k}}\bC^*$ isomorphic to the blowup of $X_{k-1}\sslash_{\mathcal{L}_{k-1}}\bC^*$ along $V_{k-1}\sslash_{\mathcal{L}_{k-1}}\bC^*$. This action obtained via Kirwan's considerations \cite{kirwanblowup}*{\S 3} coincides with the action constructed in \Cref{L:action} since they agree on the complement of proper transform of the diagonals in $X_k$. Since $H_{k-1}$ is fixed, $X_k\sslash_{\mathcal{L}_{d_k}}\bC^*$ is isomorphic to the blowup of $H_{k-1}$ along $V_{k-1}$, which is precisely the proper transform $H_k$ of $H$ in $X_k$. Since the proper transform of $H$ in $\cA_n$ is isomorphic to $\bcg_n$, we have constructed a $\bC^*$-linearized ample line bundle $\cL$ on $\cA_n$ such that $\cA_n\sslash_{\cL} \bb{C}^*$ is isomorphic to $\cB_n$.
\end{proof}
\bibliography{refs}
\bibliographystyle{plain}

\end{document}